\documentclass[twoside,leqno,10pt, A4]{amsart}
\usepackage{amsfonts}
\usepackage{amsmath}
\usepackage{amscd}
\usepackage{amssymb}
\usepackage{amsthm}
\usepackage{amsrefs}
\usepackage{latexsym}
\usepackage{mathrsfs}
\usepackage{bbm}
\usepackage{enumerate}
\usepackage{graphicx}

\usepackage{amsfonts}
\usepackage{amsmath}
\usepackage{amscd}
\usepackage{amssymb}
\usepackage{amsthm}
\usepackage{amsrefs}
\usepackage{latexsym}
\usepackage{mathrsfs}
\usepackage{bbm}
\usepackage{amscd}
\usepackage{amssymb}
\usepackage{amsthm}
\usepackage{amsrefs}
\usepackage{latexsym}
\usepackage{mathrsfs}
\usepackage{bbm}
\usepackage{enumerate}
\usepackage{graphicx}
\usepackage{color}
\begin{document}

\newtheorem{theorem}[subsection]{Theorem}
\newtheorem{proposition}[subsection]{Proposition}
\newtheorem{lemma}[subsection]{Lemma}
\newtheorem{corollary}[subsection]{Corollary}
\newtheorem{conjecture}[subsection]{Conjecture}
\newtheorem{prop}[subsection]{Proposition}
\numberwithin{equation}{section}
\newcommand{\mr}{\ensuremath{\mathbb R}}
\newcommand{\mc}{\ensuremath{\mathbb C}}
\newcommand{\dif}{\mathrm{d}}
\newcommand{\intz}{\mathbb{Z}}
\newcommand{\ratq}{\mathbb{Q}}
\newcommand{\natn}{\mathbb{N}}
\newcommand{\comc}{\mathbb{C}}
\newcommand{\rear}{\mathbb{R}}
\newcommand{\prip}{\mathbb{P}}
\newcommand{\uph}{\mathbb{H}}
\newcommand{\fief}{\mathbb{F}}
\newcommand{\majorarc}{\mathfrak{M}}
\newcommand{\minorarc}{\mathfrak{m}}
\newcommand{\sings}{\mathfrak{S}}
\newcommand{\fA}{\ensuremath{\mathfrak A}}
\newcommand{\mn}{\ensuremath{\mathbb N}}
\newcommand{\mq}{\ensuremath{\mathbb Q}}
\newcommand{\half}{\tfrac{1}{2}}
\newcommand{\f}{f\times \chi}
\newcommand{\summ}{\mathop{{\sum}^{\star}}}
\newcommand{\chiq}{\chi \bmod q}
\newcommand{\chidb}{\chi \bmod db}
\newcommand{\chid}{\chi \bmod d}
\newcommand{\sym}{\text{sym}^2}
\newcommand{\hhalf}{\tfrac{1}{2}}
\newcommand{\sumstar}{\sideset{}{^*}\sum}
\newcommand{\sumprime}{\sideset{}{'}\sum}
\newcommand{\sumprimeprime}{\sideset{}{''}\sum}
\newcommand{\sumflat}{\sideset{}{^\flat}\sum}
\newcommand{\shortmod}{\ensuremath{\negthickspace \negthickspace \negthickspace \pmod}}
\newcommand{\V}{V\left(\frac{nm}{q^2}\right)}
\newcommand{\sumi}{\mathop{{\sum}^{\dagger}}}
\newcommand{\mz}{\ensuremath{\mathbb Z}}
\newcommand{\leg}[2]{\left(\frac{#1}{#2}\right)}
\newcommand{\muK}{\mu_{\omega}}
\newcommand{\thalf}{\tfrac12}
\newcommand{\lp}{\left(}
\newcommand{\rp}{\right)}
\newcommand{\Lam}{\Lambda_{[i]}}
\newcommand{\lam}{\lambda}
\def\L{\fracwithdelims}
\def\om{\omega}
\def\pbar{\overline{\psi}}
\def\phi{\varphi}
\def\lam{\lambda}
\def\lbar{\overline{\lambda}}
\newcommand\Sum{\Cal S}
\def\Lam{\Lambda}
\newcommand{\sumtt}{\underset{(d,2)=1}{{\sum}^*}}
\newcommand{\sumt}{\underset{(d,2)=1}{\sum \nolimits^{*}} \widetilde w\left( \frac dX \right) }
\renewcommand{\a}{\alpha}

\theoremstyle{plain}
\newtheorem{conj}{Conjecture}
\newtheorem{remark}[subsection]{Remark}

\makeatletter
\def\widebreve{\mathpalette\wide@breve}
\def\wide@breve#1#2{\sbox\z@{$#1#2$}%
     \mathop{\vbox{\m@th\ialign{##\crcr
\kern0.08em\brevefill#1{0.8\wd\z@}\crcr\noalign{\nointerlineskip}%
                    $\hss#1#2\hss$\crcr}}}\limits}
\def\brevefill#1#2{$\m@th\sbox\tw@{$#1($}%
  \hss\resizebox{#2}{\wd\tw@}{\rotatebox[origin=c]{90}{\upshape(}}\hss$}
\makeatletter

\title[Lower bounds for moments of quadratic Dirichlet $L$-functions]{Lower bounds for moments of quadratic twists of modular $L$-functions}

\author[P. Gao]{Peng Gao}
\address{School of Mathematical Sciences, Beihang University, Beijing 100191, China}
\email{penggao@buaa.edu.cn}

\author[L. Zhao]{Liangyi Zhao}
\address{School of Mathematics and Statistics, University of New South Wales, Sydney, NSW 2052, Australia}
\email{l.zhao@unsw.edu.au}

\begin{abstract}
 We establish sharp lower bounds for the $2k$-th moment in the range $k \geq 1/2$ of the family of quadratic twists of modular $L$-functions at the central point.
\end{abstract}

\maketitle

\noindent {\bf Mathematics Subject Classification (2010)}: 11M06, 11F67  \newline

\noindent {\bf Keywords}: moments, quadratic twists,  modular $L$-functions, lower bounds

\section{Introduction}
\label{sec 1}

 The Birch--Swinnerton-Dyer conjecture motivates a great amount of work enriching our understanding of the central values of various $L$-functions.  In \cite{MM}, M. R. Murty and  V. K. Murty obtained an asymptotic formula for the first moment of the derivative of quadratic twists of elliptic curve $L$-functions at the central point. Their result implies that there are infinitely many such $L$-functions with nonvanishing derivatives at that point of interest.  The error term in this asymptotic formula was later improved by H. Iwaniec \cite{Iwaniec90} in a smoothed version. A better error term had also been announced earlier by D. Bump, S. Friedberg, and J. Hoffstein in \cite{BFH89}. \newline

 The result mentioned in \cite{BFH89} concerns more generally with all modular newforms. Subsequent work on the first moment of these modular $L$-functions or their derivatives at the central point can be found in \cite{Stefanicki, LR97, Radziwill&Sound, Munshi09, Munshi,Munshi11, Petrow}.  Note that when the sign of the functional equation satisfied by an $L$-function equals $-1$, then the corresponding $L$-function has a vanishing central value.  Hence its derivative at the central point is a natural object to study. \newline

 In \cite{Young1,Young2}, M. P. Young developed a recursive method to improve the error terms in the first and third moment of quadratic Dirichlet $L$-functions. This method was adapted by Q. Shen \cite{Shen21} to ameliorate the error terms in the first moment of quadratic twists of modular $L$-functions. The results of Shen can be regarded as commensurate to a result of K. Sono \cite{Sono} on the second moment of quadratic Dirichlet $L$-functions. \newline

 In \cite{S&Y}, K. Soundararajan and M. P. Young obtained a formula for the second moment of quadratic twists of modular $L$-functions at the central point under the assumption of the generalized Riemann hypothesis (GRH).  I. Petrow \cite{Petrow} computed the second moment of the derivative of quadratic twists of modular $L$-functions under GRH.  In addition to the above mentioned work on moments of modular $L$-functions at the central point, there are now well-established conjectures for these moments due to J. P. Keating and N. C. Snaith \cite{Keating-Snaith02},  with subsequent contributions from J. B. Conrey, D. W. Farmer, J. P. Keating, M. O. Rubinstein and N. C. Snaith in \cite{CFKRS} as well as from A. Diaconu, D. Goldfeld and J. Hoffstein \cite{DGH}. \newline

  In this paper, we are interested in obtaining lower bounds for the moments of modular $L$-functions at the central point of the conjectured order of magnitude. To state our results, we need some notation first.  Let $f$ be a fixed  holomorphic Hecke eigenform of weight $\kappa $ for the full modular group $SL_2 (\mathbb{Z})$. The Fourier expansion of $f$ at infinity can be written as
\[ f(z) = \sum_{n=1}^{\infty} \lambda_f (n) n^{(\kappa -1)/2} e(nz), \quad \mbox{where} \quad e(z)=\exp (2 \pi i z). \]
Here the coefficients $\lambda_f (n)$ are real and satisfy $\lambda_f (1) =1$ and $0 \neq |\lambda_f(n)|
\leq d(n)$ for $n \geq 1$ with $d(n)$ being the number of divisors of $n$. We write $\chi_d=\leg {d}{\cdot} $ for the Kronecker
symbol.  For $\Re(s) > 1$, we define the twisted modular $L$-function $L(s, f \otimes \chi_d)$ as follows.
\begin{align*}
L(s, f \otimes \chi_d) &= \sum_{n=1}^{\infty} \frac{\lambda_f(n)\chi_d(n)}{n^s}
 = \prod_{p\nmid d} \left(1 - \frac{\lambda_f (p) \chi_d(p)}{p^s}  + \frac{1}{p^{2s}}\right)^{-1}.
\end{align*}

 Then $L(s, f \otimes \chi_d)$ has analytical continuation to the entire complex plane and satisfies the functional equation
\begin{align*}
\Lambda (s, f \otimes \chi_d) = \left(\frac{|d|}{2\pi} \right)^s \Gamma (s + \tfrac{\kappa -1}{2}) L(s, f \otimes \chi_d)
= i^\kappa \epsilon(d ) \Lambda (1- s, f \otimes \chi_d),
\end{align*}
where $\epsilon(d) =1$ if $d$ is positive, and  $\epsilon(d) =-1$ if $d$ is negative. \newline

We consider the family $\{ L(s, f \otimes \chi_{8d}) \}$ with $d$ positive, odd and square-free. Note that
the sign of the functional equation for each member in this family is positive if $\kappa \equiv 0 \pmod 4$, in which case one can study moments of these central $L$-values.  It is expected that for all positive real $k$, we have
\begin{align*}
  \sumstar_{\substack{ 0<d<X \\ (d,2)=1}}|L(\tfrac{1}{2}, f \otimes \chi_{8d})|^k \sim C_kX(\log X)^{(k(k-1))/2},
\end{align*}
  where $\Sigma^*$ means that the sum runs over square-free integers and the $C_k$'s are explicitly computable constants. \newline

When $\kappa \equiv 2 \pmod 4$, the sign of the functional equation for each member of the above family is negative so that the corresponding central $L$-value is zero. In this case, one can study moments of the derivatives of these $L$-functions at $s=1/2$. The expectation in this case is,  for all positive real $k$,
\begin{align} \label{upperlower}
  \sumstar_{\substack{ 0<d<X \\ (d,2)=1}}|L'(\tfrac{1}{2}, f \otimes \chi_{8d})|^k \sim C'_kX(\log X)^{(k(k+1))/2},
\end{align}
with explicitly computable constants $C'_k$. \newline

  The aim of this paper is to establish the lower bounds, of the conjectured order of magnitudes, for the moments in \eqref{upperlower}. Our results are as follows.
\begin{theorem}
\label{thmlowerbound}
   With notations as above and let $k \geq 1$. For $\kappa \equiv 0 \pmod 4$ and $\kappa \neq 0$, we have
\begin{equation} \label{k=0mod4}
   \sumstar_{\substack{ 0<d<X \\ (d,2)=1}}|L(\tfrac{1}{2},f \otimes \chi_{8d})|^{k} \gg_k X(\log X)^{(k(k-1))/2}.
\end{equation}
   For $\kappa \equiv 2 \pmod 4$, we have
\begin{equation} \label{k=2mod4}
   \sumstar_{\substack{ 0<d<X \\ (d,2)=1}}|L'(\tfrac{1}{2},f \otimes \chi_{8d})|^{k} \gg_k X(\log X)^{(k(k+1))/2}.
\end{equation}
\end{theorem}

   Our proof of the above theorem is mainly based on the lower bounds principle of Heap and Soundararajan in \cite{H&Sound}.  As this method requires the evaluation of the twisted moments of modular functions, we shall also make crucial use of the results of Q. Shen \cite{Shen21} on the twisted first moment of twisted modular functions.

\section{Plan of the Proof}
\label{sec 2'}

As the proofs of \eqref{k=0mod4} and \eqref{k=2mod4} are similar, we shall only consider the case $\kappa \equiv 2 \pmod 4$ in Theorem \ref{thmlowerbound} throughout the paper.  We  shall further replace $k$ by $2k$ and assume that $k > 1/2$ since the case $k=1/2$ follows from \cite[Theorems 1.1--1.2]{Shen21}.  Let $\Phi$ be a smooth, non-negative function compactly supported on $[1/8,7/8]$ such that
$\Phi(x) \leq 1$ for all $x$ and $\Phi(x) =1$ for $x\in [1/4,3/4]$. Assuming that $X$ is a large positive real number, we find that in order to prove \eqref{k=2mod4}, it suffices to show that
\begin{equation} \label{lowerbound}
   \sumstar_{\substack{(d,2)=1}}|L'(\tfrac{1}{2},f \otimes \chi_{8d})|^{2k}\Phi \left( \frac dX \right) \gg_k X(\log X)^{\frac{2k(2k+1)}{2}}.
\end{equation}
  To this end, we recall the lower bounds principle of Heap and Soundararajan in \cite{H&Sound} for our situation. Let $\{ \ell_j \}_{1 \leq j \leq R}$ be a sequence of even
  natural numbers such that $\ell_1= 2\lceil N \log \log X\rceil$ and $\ell_{j+1} = 2 \lceil N \log \ell_j \rceil$ for $j \geq 1$, where $N, M$ are two large natural numbers depending on $k$ only and $R$
  is the largest natural number satisfying $\ell_R >10^M$. We shall choose $M$ large enough to ensure $\ell_{j} >\ell_{j+1}^2$ for all $1 \leq j \leq R-1$. It follows that
\begin{align}
\label{sumoverell}
  R \ll \log \log \ell_1 \quad \mbox{and} \quad \sum^R_{j=1}\frac 1{\ell_j} \leq \frac 2{\ell_R}.
\end{align}

    Let ${P}_1$ be the set of odd primes not exceeding $X^{1/\ell_1^2}$ and
${P_j}$ be the set of primes in the interval $(X^{1/\ell_{j-1}^2}, X^{1/\ell_j^2}]$ for $2\le j\le R$, we define
\begin{equation*}
{\mathcal P}_j(d) = \sum_{p\in P_j} \frac{\lambda_f(p)}{\sqrt{p}} \chi_{8d}(p),
\end{equation*}
and for any $\alpha \in \mr$,
\begin{align}
\label{defN}
{\mathcal N}_j(d, \alpha) = E_{\ell_j} (\alpha {\mathcal P}_j(d)) \quad \mbox{and} \quad \mathcal{N}(d, \alpha) = \prod_{j=1}^{R} {\mathcal N}_j(d,\alpha),
\end{align}
  where for any integer $\ell \geq 0$ and any $x \in \mr$,
\begin{equation*}
E_{\ell}(x) = \sum_{j=0}^{\ell} \frac{x^{j}}{j!}.
\end{equation*}

  Since $\lambda_f(p)$ is real, so is ${\mathcal P}_j(d)$. Consequently, it follows from \cite[Lemma 1]{Radziwill&Sound} that the
  quantities defined in \eqref{defN} are all positive. As $2k>1$, H\"older's inequality gives
\begin{align}
\label{basicbound1}
\begin{split}
 \sumstar_{(d,2)=1} L' &(\half, f \otimes \chi_{8d}) \mathcal{N}(d, 2k-1)   \Phi\Big( \frac dX \Big) \\
 \leq & \Big ( \sumstar_{(d,2)=1}|L'(\half, f \otimes \chi_{8d})|^{2k} \Phi \Big( \frac dX \Big)\Big )^{1/(2k)}\Big ( \sumstar_{(d,2)=1}  \mathcal{N}(d, 2k-1)^{2k/(2k-1)}
 \Phi \Big( \frac dX \Big) \Big)^{(2k-1)/(2k)}.
\end{split}
\end{align}

We deduce from \eqref{basicbound1} that in order to establish \eqref{lowerbound} and hence \eqref{k=2mod4}, it suffices to establish the following propositions.
\begin{proposition}
\label{Prop4} With notations as above, we have, for $\kappa \equiv 2 \pmod 4$ and $k > 1/2$,
\begin{align}
\label{LprimeN}
\sumstar_{(d,2)=1} L'(\half, f \otimes \chi_{8d}){\mathcal N}(d, 2k-1)\Phi\Big(\frac{d}{X}\Big) \gg X (\log X)^{((2k)^2+1)/2}.
\end{align}
\end{proposition}

\begin{proposition}
\label{Prop5} With notations as above, we have, for $\kappa \equiv 2 \pmod 4$ and $k > 1/2$,
\begin{align*}
\sumstar_{(d,2)=1}\mathcal{N}(d, 2k-1)^{2k/(2k-1)} \Phi\Big(\frac{d}{X}\Big) \ll X ( \log X  )^{(2k)^2/2}.
\end{align*}
\end{proposition}
The remainder of the paper is thus devoted to the proofs of these propositions.

\section{Preliminaries}
\label{sec 2}

   We reserve the letter $p$ for a prime number and cite the following result on sum over primes.
\begin{lemma}
\label{RS} Let $x \geq 2$. We have, for some constant $b$,
\begin{align} \label{lam2p}
\sum_{p\le x} \frac{\lambda^2_f(p)}{p} =& \log \log x + b+ O\Big(\frac{1}{\log x}\Big).
\end{align}
Moreover, we have
 \begin{equation} \label{logp}
\sum_{p\le x} \frac {\log p}{p} = \log x + O(1).
\end{equation}
\end{lemma}
\begin{proof} The assertion in \eqref{lam2p} follows from the Rankin-Selberg theory for $L(s, f)$,  which can be found in  \cite[Chapter 5]{iwakow}. The formula in \eqref{logp} is given in \cite[Lemma 2.7]{MVa1}. \end{proof}

Set $\delta_{n=\square}$ to be $1$ if $n$ is a square and $0$ otherwise.  We note the following result that can be established in a manner similar to \cite[Proposition 1]{Radziwill&Sound}.
\begin{lemma} \label{PropDirpoly}  For large $X$ and any odd positive integer $n$, we have
\begin{align}
\label{meancharsum}
\sumstar_{\substack{(d,2)=1}} \chi_{8d}(n) \Phi\Big(\frac{d}{X}\Big)=
\displaystyle \delta_{n=\square}{\widehat \Phi}(1) \frac{2X}{3\zeta(2)} \prod_{p|n} \Big(\frac p{p+1}\Big) + O(X^{1/2+\varepsilon} \sqrt{n}
).
\end{align}
\end{lemma}

  For any complex number $s$, let ${\widehat \Phi}(s)$ be the Mellin transform of the function $\Phi$ described in Section \ref{sec 2'}, i.e.
\begin{equation*}
{\widehat \Phi}(s) = \int\limits_{0}^{\infty} \Phi(x)x^{s}\frac {dx}{x}.
\end{equation*}

 Applying \cite[Theorem 1.4]{Shen21}, we have the following asymptotic formulas for the twisted first moment of quadratic modular $L$-functions.
\begin{lemma}
\label{Prop1}
  Using the same notations as above and writing any odd $l$ as  $l=l_1l^2_2$ with $l_1$ square-free, we have, for $\kappa \equiv 0 \pmod 4$ and any
  $\varepsilon>0$,
\begin{align}
\label{eq:1stmoment}
\begin{split}
 \sumstar_{(d,2)=1}L(\tfrac{1}{2},f \otimes \chi_{8d})\chi_{8d}(l)\Phi \Big( \frac dX \Big)
=&  \frac {8C }{\pi^2}  \widehat{\Phi}(1)\frac {\lambda_f(l_1)}{\sqrt{l_1}g(l)}L(1, \operatorname{sym}^2 f)X +O \left(X^{3/4+\varepsilon}l^{1/2+\varepsilon}\right ),
\end{split}
\end{align}
     where $g(l)$ is a multiplicative function, i.e. $g(l)=\displaystyle \prod_{p^i \parallel l}g(p^i)$. Here $g(p^i)=g(p) \geq 1$ for all $i \geq 1$ and $1/g(p)=1+O(1/p)$.  Moreover, for $\kappa \equiv 2 \pmod 4$, we have
\begin{align}
\label{eq:1stmomentder}
\begin{split}
 \sumstar_{(d,2)=1}L'(\tfrac{1}{2},f \otimes \chi_{8d})\chi_{8d}(l)\Phi \Big( \frac dX \Big)
=&  C \widehat{\Phi}(1)\frac {\lambda_f(l_1)}{\sqrt{l_1}g(l)}X \Big (\log \frac {X}{l_1}+C_2+\sum_{\substack{p | l}} \frac {C_2(p)}{p}\log
p \Big )+O \left(X^{3/4+\varepsilon}l^{1/2+\varepsilon}\right ),
\end{split}
\end{align}
  where  $C$ is an absolute constant, $C_2$ is a constant whose value depends only on $\Phi$ and $C_2(p) \ll 1$ for all $p$.
\end{lemma}
\begin{proof}

  We write $l = l_1 l_2^2$ with $l_1$ square-free for any positive, odd integer $l$. For any complex number $\alpha$, we define
\begin{align*}
M(\alpha, l) = \sumstar_{(d,2)=1}L(\tfrac{1}{2} + \alpha,f \otimes \chi_{8d})\chi_{8d}(l)\Phi \Big( \frac dX \Big).
\end{align*}
 It follows from \cite[Conjecture 1.3, Theorem 1.4]{Shen21} and the proofs of Theorems 1.1 and 1.2 in \cite{Shen21} that, for
 \[ |\Re(\a)| \ll (\log X)^{-1}, \quad |\Im(\a)| \ll (\log X)^2, \]
 we have
\begin{align}
\label{assump}
\begin{split}
M(\alpha, l)  = \frac{4X {\widehat \Phi}(1)}{\pi^2 l_1 ^{1/2 + \a}} L(1+2\alpha, \operatorname{sym}^2 f)  &Z (\tfrac{1}{2} + \alpha, l)+i^{\kappa} \frac{4
 \gamma_\a X^{1-2\a} {\widehat \Phi}(1-2\a)}{\pi^2 l_1 ^{1/2 - \a}} L(1-2\alpha, \operatorname{sym}^2 f)  Z (\tfrac{1}{2} - \alpha,l)
 \\
 &+ O(l^{1/2 + \varepsilon}X^{1/2+\varepsilon}),
\end{split}
\end{align}
 where the implied constant depends on $\varepsilon$, $h$ and $\Phi$,
\begin{align*}
\gamma_\alpha = \left( \frac{8}{2\pi}\right)^{-2\alpha} \frac{\Gamma(\frac{\kappa}{2} - \alpha)}{\Gamma (\frac{\kappa}{2} + \alpha)}.
\end{align*}
and $Z(\tfrac{1}{2}+\alpha,l)$ is defined below.  Note that though dubbed a conjecture, \cite[Conjecture 1.3]{Shen21} is proved therein.  Moreover, $Z (\tfrac{1}{2} + \gamma, l)$ is analytic and absolutely convergent in the region $\Re(\gamma) > -\frac{1}{4}$ such that in this range, we write
 \[ L(1+2\alpha, \operatorname{sym}^2 f)  Z (\tfrac{1}{2} + \alpha, l) :=\lambda_f(l_1)\prod_{(p,2)=1}  Z_{p} (\tfrac{1}{2}+\alpha,  l), \]
 where
 \begin{align}
\label{equ:conj1}
\begin{split}
  Z_{p} (\tfrac{1}{2} + \gamma,  l) =\left\{
 \begin{array}  [c]{ll}
\displaystyle{ \lambda_f(p)^{-1} p^{1/2 + \gamma} \left( \frac{p}{p+1} \right) \left( \frac{1}{2} \left( 1- \frac{\lambda_f (p)}{p^{1/2 + \gamma} } +
  \frac{1}{p^{1+2\gamma}} \right)^{-1}- \frac{1}{2 } \left( 1+ \frac{\lambda_f (p)}{p^{1/2 + \gamma} }+ \frac{1}{p^{1+2\gamma}} \right) ^{-1} \right),}
  & \text{if} \; p| l_1 ,\\ \\
\displaystyle{   \frac{p}{p+1}\left( \frac{1}{2} \left( 1- \frac{\lambda_f (p)}{p^{1/2 + \gamma} } + \frac{1}{p^{1+2\gamma}} \right) ^{-1}+ \frac{1}{2 } \left( 1+ \frac{\lambda_f (p)}{p^{1/2 + \gamma} }+ \frac{1}{p^{1+2\gamma}} \right) ^{-1} \right),} & \text{if} \; p\nmid l_1, \; p| l_2,\\ \\
\widetilde Z_p(\tfrac{1}{2} + \gamma), &\text{if} \; (p, 2l)=1,
\end{array}\right.
\end{split}
\end{align}
  and where
\begin{align}
\begin{split}
\label{equ:conj11}
\widetilde Z_p(\tfrac{1}{2} + \gamma) :=\displaystyle{   1 +  \frac{p}{p+1}\left( \frac{1}{2} \left( 1- \frac{\lambda_f (p)}{p^{1/2 + \gamma} } + \frac{1}{p^{1+2\gamma}} \right)^{-1}+ \frac{1}{2 } \left( 1+ \frac{\lambda_f (p)}{p^{1/2 + \gamma} }+ \frac{1}{p^{1+2\gamma}} \right) ^{-1} -1 \right)}.
\end{split}
\end{align}

  Note that we have $L(1+2\alpha, \operatorname{sym}^2 f) =\prod_pL_p(1+2\alpha, \operatorname{sym}^2 f)$ with
\begin{align*}
\begin{split}
L_p(1+2\alpha, \operatorname{sym}^2 f) := \left( 1- \frac{\lambda_f (p)}{p^{1/2 + \alpha} } + \frac{1}{p^{1+2\alpha}} \right)^{-1}\left( 1+ \frac{\lambda_f (p)}{p^{1/2 + \alpha} } + \frac{1}{p^{1+2\alpha}} \right)^{-1}\left( 1-  \frac{1}{p^{1+2\alpha}} \right)^{-1}.
\end{split}
\end{align*}

  If $\kappa \equiv 0 \pmod 4$, we set $\alpha=0$ in \eqref{assump} to compute $Z(\half, l)$ using \eqref{equ:conj1} and the above to deduce \eqref{eq:1stmoment} with $C=\prod_p C_p$, where
\begin{align*}
\begin{split}
 C_p=
\begin{cases}
  L_2(1, \operatorname{sym}^2 f)^{-1}, & \quad p=2, \\ \\
  \widetilde Z_p(\tfrac{1}{2})L_p(1, \operatorname{sym}^2 f)^{-1}, & \quad p>2.
\end{cases}
\end{split}
\end{align*}
  Note that it follows from the proof of \cite[Lemma 2.6]{Shen21} that $C_p=1+O(p^{-1-\varepsilon})$ for some $\varepsilon>0$, so that $C$ is an absolute constant. \newline

  Moreover, it is readily seen that $g(l)$ is a multiplicative function such that $g(l)=\displaystyle \prod_{p^i\parallel l}g(p^i)$, where for all $i \geq 1$,
\begin{align*}
g(p^i) = Z_{p} (\tfrac{1}{2},  l)^{-1}\widetilde Z_p(\tfrac{1}{2}).
\end{align*}
  Simplifying the expressions in \eqref{equ:conj1} and \eqref{equ:conj11} leads to
 \begin{align*}
\begin{split}
  Z_{p} (\tfrac{1}{2},  l) =&\left\{
 \begin{array}  [c]{ll}
\frac p{p+1} \left( 1- \frac{\lambda_f (p)}{p^{1/2} } + \frac{1}{p} \right)^{-1} \left( 1+ \frac{\lambda_f (p)}{p^{1/2} }+ \frac{1}{p} \right)^{-1} ,
  & \text{if} \; p| l_1 ,\\ \\
\left( 1- \frac{\lambda_f (p)}{p^{1/2} } + \frac{1}{p} \right)^{-1} \left( 1+ \frac{\lambda_f (p)}{p^{1/2} }+ \frac{1}{p} \right)^{-1}, & \text{if} \; p\nmid l_1, \; p| l_2,
\end{array}\right. \quad \mbox{and} \\
\widetilde Z_p(\tfrac{1}{2}) =&  \left( 1- \frac{\lambda_f (p)}{p^{1/2} } + \frac{1}{p} \right)^{-1} \left( 1+ \frac{\lambda_f (p)}{p^{1/2} }+ \frac{1}{p} \right)^{-1} +\frac 1{p+1} .
\end{split}
\end{align*}
Hence, we deduce that
\begin{align*}
\begin{split}
  g(p^i) = &\left\{
 \begin{array}  [c]{ll}
\frac {p+1}p\left ( 1+\frac 1{p+1}\left( \left( 1+ \frac{1}{p}\right)^2- \frac{\lambda^2_f (p)}{p}  \right)\right ) ,
  & \text{if} \; p| l_1 ,\\ \\
 1+\frac 1{p+1}\left( \left( 1+ \frac{1}{p}\right)^2- \frac{\lambda^2_f (p)}{p}  \right) , & \text{if} \; p\nmid l_1, \; p| l_2,
\end{array}\right. 
\end{split}
\end{align*}  
  One then checks from the above that $g(p^i) \geq 1$ using $|\lambda_f(p)|<2$. Similarly, one shows that $1/g(p)=1+O(1/p)$ using \eqref{equ:conj1}, \eqref{equ:conj11} and the above. \newline

If $\kappa \equiv 2 \pmod 4$,  we differentiate both sides of \eqref{assump} with respect to $\a$.  The contribution of the derivative of the error
   term is still $O(l^{1/2 + \varepsilon}X^{1/2+\varepsilon})$ using Cauchy's integral formula and the observation that the error
   term in \eqref{assump} is holomorphic on the disc centred at $(0,0)$ with radius $\ll (\log X)^{-1}$ from the proof of \cite[Theorem
   1.4]{Shen21}. Upon setting $\alpha=0$, we obtain that
\begin{align*}
\begin{split}
\sumstar_{(d,2)=1} L'(\tfrac{1}{2},f \otimes \chi_{8d})\chi_{8d}(l)\Phi\Big(\frac dX \Big) = \frac{8 \widehat{\Phi}(1)}{\pi^2l_1^{1/2}} L(1, \operatorname{sym}^2 f) & Z(\half, l)  X
\Big(
\log \frac {X}{l_1}
 + 2 \frac{L'(1, \operatorname{sym}^2 f) }{L(1, \operatorname{sym}^2 f) }
+ \frac{ Z'(\half, l)}{ Z(\half, l)}\\
&\quad + \log \frac{8}{2\pi}
+ \frac{\Gamma'(\frac{\kappa}{2})}{\Gamma(\frac{\kappa}{2})}
+ \frac{\widehat{\Phi}'(1)}{\widehat{\Phi}(1)}
\Big)+ O(l^{1/2 + \varepsilon}X^{1/2+\varepsilon}).
\end{split}
\end{align*}

Computing of  $Z(\half, l)$ and  $Z'(\half, l)$ using \eqref{equ:conj1} then leads to \eqref{eq:1stmomentder} and completes the proof.
\end{proof}

 The following result is analogous to \cite[Lemma 1]{H&Sound} and is needed in the proof of Proposition \ref{Prop5}.
\begin{lemma}
\label{lem1}
 For $1 \le j\le R$, we have
$$
{\mathcal N}_j(d, 2k-1)^{2k/(2k-1)}  \le
{\mathcal N}_j(d, 2k) \frac{ \left( 1+e^{-\ell_j} \right )^{2k/(2k-1)} }{\left( 1-e^{-\ell_j} \right )^2} + {\mathcal Q}_j(d),
$$
where
$$
{\mathcal Q}_j(d) =\Big( \frac{124k^2 {\mathcal P}_j(d)}{\ell_j} \Big)^{2r_k\ell_j},  \quad \mbox{with} \quad r_k=1+\lceil \frac {k}{2k-1} \rceil.
$$
\end{lemma}
\begin{proof}
  As in the proof of \cite[Lemma 3.4]{Gao2021-3}, we have for $|z| \le aK/20$ with $0<a \leq 2$,
\begin{align}
\label{Ebound}
\Big| \sum_{r=0}^K \frac{z^r}{r!} - e^z \Big| \le \frac{e^{|z|} |z|^{K}}{K!} \le \Big(\frac{a e}{20}\Big)^{K}.
\end{align}
  By taking $z=\alpha {\mathcal P}_j(d), K=\ell_j $ and $a=\min (|\alpha|, 2 )$ in \eqref{Ebound}, we see that when $|{\mathcal P}_j(d)| \le \ell_j/(20(1+|\alpha|))$,
\begin{align*}
{\mathcal N}_j(d, \alpha) & = \exp (\alpha \mathcal{P}_j(d)) + \sum_{r=0}^{l_j} \frac{\left( \alpha \mathcal{P}_j(d)  \right)^{r}}{r!} - \exp (\alpha \mathcal{P}_j(d))   \\
 & \leq  \exp ( \alpha{\mathcal P}_j(d) )\left( 1+ \exp ( |\alpha {\mathcal P}_j(d)| ) \left( \frac{a e}{20} \right)^{ \ell_j}  \right) \leq \exp ( \alpha {\mathcal P}_j(d)  ) \left( 1+   e^{-\ell_j}  \right).
\end{align*}

  Similarly, we have
\begin{align*}
{\mathcal N}_j(d, \alpha) \geq & \exp ( \alpha{\mathcal P}_j(d) )\left( 1- \exp ( |\alpha {\mathcal P}_j(s)| ) \left( \frac{a e}{20} \right)^{ \ell_j} \right) \geq \exp ( \alpha {\mathcal P}_j(d)  ) \left( 1-   e^{-\ell_j}  \right).
\end{align*}

  We apply the above estimations to ${\mathcal N}_j(d, 2k-1)$, ${\mathcal N}_j(d, k)$ to get that if $k>1/2$ and $|{\mathcal P}_j(d)| \le \ell_j /(60k)$ , then
\begin{align} \label{est1'}
|\mathcal{N}_j(d,
 2k-1)|^{\frac {2k}{2k-1}}
\leq \exp ( 2k {\mathcal P}_j(d)  ) \left( 1+  e^{-\ell_j}  \right)^{\frac {2k}{2k-1}}
 \leq |{\mathcal N}_j(d, 2k)|^2 \left( 1+e^{-\ell_j} \right )^{\frac {2k}{2k-1}}\left( 1-e^{-\ell_j} \right )^{-2}.
\end{align}

 Moreover, if $|{\mathcal P}_j(d)| > \ell_j/(60k)$,
\begin{align}
\label{4.2}
\begin{split}
|{\mathcal N}_j(d, 2k-1)| &\le \sum_{r=0}^{\ell_j} \frac{|(2k-1){\mathcal P}_j(d)|^r}{r!} \le
|2k{\mathcal P}_j(d)|^{\ell_j} \sum_{r=0}^{\ell_j} \Big( \frac{60k}{\ell_j}\Big)^{\ell_j-r} \frac{1}{r!}   \le \Big( \frac{124k^2 |{\mathcal
P}_j(d)|}{\ell_j}\Big)^{\ell_j} .
\end{split}
\end{align}
  The assertion of the lemma now follows from \eqref{est1'} and \eqref{4.2} and the observation that ${\mathcal
P}_j(d)$ is real and $\ell_j$ is even.
\end{proof}

\section{Proof of Proposition \ref{Prop4}}
\label{sec 4}

   We define functions $b_j(n), 1 \leq j \leq R$  such that $b_j(n) \in \{ 0,1\}$ and $b_j(n)=1$ if and only if $n$ has all prime factors in $P_j$ such that $\Omega(n) \leq \ell_j$, where $\Omega(n)$ denotes the number of primes dividing $n$. We also define $w(n)$ to be the multiplicative function such that $w(p^{\alpha}) = \alpha!$ for prime powers $p^{\alpha}$.   Also write $\widetilde{\lambda}_f(n)$ for the completely multiplicative function defined on primes $p$ by $\widetilde{\lambda}_f(p)=\lambda_f(p)$.  Using these notations, we arrive at
\begin{equation} \label{5.1}
{\mathcal N}_j(d, 2k-1) = \sum_{n_j} \frac{\widetilde{\lambda}_f(n_j)}{\sqrt{n_j}} \frac{(2k-1)^{\Omega(n_j)}}{w(n_j)}  b_j(n_j) \chi_{8d}(n_j), \quad 1\le j\le R.
\end{equation}
    Observe that each ${\mathcal N}_j(d, 2k-1)$ is a short Dirichlet polynomial as $b_j(n_j)=0$ unless $n_j \leq
    (X^{1/\ell_j^2})^{\ell_j}=X^{1/\ell_j}$. It follows from this and \eqref{sumoverell} that  ${\mathcal N}(d, 2k-1)$ is also a short Dirichlet polynomial of length at most $X^{1/\ell_1+ \ldots +1/\ell_R} < X^{2/10^{M}}$. \newline

 We expand the term ${\mathcal N}(d, 2k-1)$ using \eqref{5.1} and apply Lemma  \ref{Prop1} to evaluate the left-hand side of \eqref{LprimeN}.  In this process, we may ignore the contribution from the error term in Lemma \ref{Prop1} as ${\mathcal N}(d, 2k-1)$
 is a short Dirichlet polynomial. Thus we focus on the main term contribution. Writing $n_j=(n_j)_1(n_j)_{2}^2$ with $(n_j)_{1}$ square-free, we get that
\begin{align*}
 \sumstar_{(d,2)=1} L'(\half, f \otimes \chi_{8d}) {\mathcal N}(d, 2k-1)\Phi\Big(\frac{d}{X}\Big)  \gg & X   \sum_{n_1, \cdots, n_{R} }  \Big( \prod_{j=1}^{R} \frac{\widetilde{\lambda}_f(n_j)\widetilde{\lambda}_f((n_j)_1)}{\sqrt{n_j(n_j)_{1}}}
\frac{(2k-1)^{\Omega(n_j)}}{w(n_j) } b_j(n_j) \frac {1}{g(n_{j})}  \Big) \\
 & \hspace*{1in} \times \Big (\log \Big ( \frac {X}{(n_1)_1 \cdots (n_{R})_1} \Big
)+C_2+\sum_{\substack{p | n_1 \cdots n_{R} }} \frac {C_2(p)}{p}\log p \Big ).
\end{align*}
We shall concentrate on the terms involving $\log (X/((n_1)_1 \cdots (n_R)_1))$.   The other terms involving
\[ C_2+\sum_{\substack{p | n_1 \cdots n_{R} }} \frac {C_2(p)}{p}\log p\]
can be shown, using the same argument in this section, to be
 \[ \ll  X (\log X)^{((2k)^2+1)/2-1} \]
 and hence negligible. Thus we deduce that
\begin{align*}
  \sumstar_{(d,2)=1} L'(\half, f \otimes \chi_{8d}){\mathcal N}(d, 2k-1)\Phi\Big(\frac{d}{X}\Big)
\gg S_1-S_2,
\end{align*}
   where
\[ S_1=   X \log X   \sum_{n_1, \cdots, n_{R} }  \Big( \prod_{j=1}^{R} \frac{\widetilde{\lambda}_f(n_j)\widetilde{\lambda}_f((n_j)_1)}{\sqrt{n_j(n_j)_{1}}} \frac{(2k-1)^{\Omega(n_j)}}{w(n_j) } b_j(n_j) \frac {1}{g(n_{j})}  \Big)  \]
and
\[ S_2=   X   \sum_{n_1, \cdots, n_{R} }  \Big( \prod_{j=1}^{R} \frac{\widetilde{\lambda}_f(n_j)\widetilde{\lambda}_f((n_j)_1)}{\sqrt{n_j(n_j)_{1}}}
\frac{(2k-1)^{\Omega(n_j)}}{w(n_j) } b_j(n_j) \frac {1}{g(n_{j})}  \Big) \log \big ( \prod^R_{i=1} (n_i)_1  \big ). \]
   It remains to estimate $S_1$ from below and $S_2$ from above. We bound $S_1$ first by recasting it as
\begin{align*}
 S_1= &  X \log X   \prod_{j=1}^{R} \sum_{n_j }  \Big (\frac{\widetilde{\lambda}_f(n_j)\widetilde{\lambda}_f((n_j)_1)}{\sqrt{n_j(n_j)_{1}}}
\frac{(2k-1)^{\Omega(n_j)}}{w(n_j) } b_j(n_j) \frac {1}{g(n_{j})}  \Big).
\end{align*}
  We consider the sum above over $n_j$ for a fixed $j$. Note that the factor $b_j(n_j)$ restricts $n_j$ to have all prime
  factors in $P_j$ such that $\Omega(n_j) \leq \ell_j$. If we remove the restriction on $\Omega(n_j)$, then we use the fact that $g(p^i)>0$ from Lemma \ref{Prop1} and $k>1/2$ that the sum becomes
\begin{align}
\label{6.02'}
\begin{split}
\prod_{\substack{p\in P_j }} & \left( \sum_{i=0}^{\infty} \frac{\lambda^{2i}_f(p)}{p^i} \frac{(2k-1)^{2i}}{(2i)!g(p^{2i})} + \sum_{i=0}^{\infty}
\frac{\lambda^{2i+2}_f(p)}{p^{i+1}}
\frac{(2k-1)^{2i+1}}{(2i+1)!}\frac {1}{g(p^{2i+1})} \right) 
\geq  \prod_{\substack{p\in P_j }}\Big( 1+\Big( \frac {(2k-1)^{2}}{2}+2k-1 \Big) \frac {\lambda^{2}_f(p)}{pg(p)} \Big) \\
& =  \exp \Big( \sum_{\substack{p\in P_j}} \log \Big( 1+\Big( \frac {(2k-1)^{2}}{2}+2 k-1 \Big) \frac {\lambda^{2}_f(p)}{pg(p)} \Big)
\geq  \exp \Big( \Big( \frac {(2k-1)^{2}}{2}+2 k-1 \Big) \sum_{\substack{p\in P_j}} \frac {\lambda^{2}_f(p)}p \Big)\times D_j,
\end{split}
\end{align}
  where the last inequality above follows from the observation that $\log (1+x)=x+O(x^2)$ for all $x>0$ and $1/g(p)=1+O(1/p)$ from Lemma \ref{Prop1}. Here
  \[ D_j= \exp \Big( \sum_{\substack{p\in P_j }} O\Big( \frac 1{p^2}\Big )\Big )>0. \]

  On the other hand, using Rankin's trick, noting that $2^{\Omega(n_1)-\ell_1}\ge 1$ if $\Omega(n_1) > \ell_1$,  we conclude that the error
  introduced this way does not exceed
\begin{equation} \label{rankintr}
\begin{split}
 \sum_{n_j} \frac{|\widetilde{\lambda}_f(n_j)\widetilde{\lambda}_f((n_j)_1)|}{\sqrt{n_j(n_{j})_{1}}} &
\frac{(2k-1)^{\Omega(n_j)}}{w(n_j) } 2^{\Omega(n_j)-\ell_j} \frac {1}{g(n_{j})} \\
\le & 2^{-\ell_j} \prod_{\substack{p\in P_j }} \Big( 1+\sum_{i=1}^{\infty} \frac{\lambda^{2i}_f(p)}{p^i} \frac{(2k-1)^{2i}2^{2i}}{(2i)!} \frac {1}{g(p^{2i})}
+ \sum_{i=0}^{\infty} \frac{\lambda^{2i+2}_f(p)}{p^{i+1}}
\frac{(2k-1)^{2i+1}2^{2i+1}}{(2i+1)!}\frac {1}{g(p^{2i+1})}\Big).
\end{split}
\end{equation}
Now using the well-known bound
\begin{equation} \label{1+xexpbound}
1+x \leq e^x, \quad \mbox{for all} \; x \in \rear,
\end{equation}
the expression in \eqref{rankintr} is
\begin{align} \label{errorbound}
\begin{split}
\le & 2^{-\ell_j} \exp \Big ( \sum_{\substack{p\in P_j }} \Big( \sum_{i=1}^{\infty} \frac{\lambda^{2i}_f(p)}{p^i} \frac{(2k-1)^{2i}2^{2i}}{(2i)!} \frac {1}{g(p^{2i})}  +
\sum_{i=0}^{\infty} \frac{\lambda^{2i+2}_f(p)}{p^{i+1}}
\frac{(2k-1)^{2i+1}2^{2i+1}}{(2i+1)!}\frac {1}{g(p^{2i+1})}\Big) \Big ) \\
\leq & 2^{-\ell_j}
\exp\Big( \big (2(2k-1)^2+2(2k-1)\big ) \sum_{p\in P_j} \frac{\lambda^{2}_f(p)}{p} +O \Big( \sum_{p\in P_j} \frac 1{p^2} \Big) \Big ).
\end{split}
\end{align}

 We may take $M$ sufficiently large so that every $\ell_j$, $1 \leq j \leq R$ is large and we may also take $N$ large enough so that by Lemma \ref{RS},
\begin{align}
\label{boundsforsumoverp}
  \frac {\ell_j}{4N}  \leq \sum_{p \in P_j}\frac {\lambda_f(p)^{2}}{p} \leq \frac 2N \ell_j, \quad 1 \leq j \leq R.
\end{align}
 It follows from \eqref{errorbound} and \eqref{boundsforsumoverp} that the error incurred in discarding the condition on $\Omega(n_j)$ is
\begin{align}
\label{lowerbound2}
\begin{split}
\leq & 2^{-\ell_j/2}\exp \Big ( \Big(\frac {(2k-1)^{2}}{2}+2 k-1 \Big )\sum_{\substack{p\in P_j}}\frac {\lambda^{2}_f(p)}p \Big )D_j.
\end{split}
\end{align}

  Combining \eqref{6.02'} and \eqref{lowerbound2}, we infer that the sum over $n_j$ for each $j$, $1 \leq j \leq R$ in the expression
  of $S_1$ is
\begin{align*}
\begin{split}
\geq & (1-2^{-\ell_j/2})\exp \Big ( \Big(\frac {(2k-1)^{2}}{2}+2 k-1 \Big )\sum_{\substack{p\in P_j}}\frac {\lambda^{2}_f(p)}p \Big )D_j.
\end{split}
\end{align*}
Hence
\begin{align}
\label{S1bound}
\begin{split}
 S_1 \geq &  X\log X \prod^R_{j=1}\Big ( (1-2^{-\ell_j/2})\exp \Big ( \Big(\frac {(2k-1)^{2}}{2}+2 k-1 \Big )\sum_{\substack{p\in
 P_j}}\frac {\lambda^{2}_f(p)}p \Big )D_j \Big ).
\end{split}
\end{align}

  Now, we estimate $S_2$ by expanding  $\log \big ( \prod^R_{i=1} (n_i)_1  \big )$ as a sum of logarithms of primes dividing $\prod^R_{i=1}
  (n_i)_1$ to obtain that
\begin{align}
\label{6.1}
\begin{split}
 S_2 \leq & X \sum_{q \in \bigcup P_j}\Big( \sum_{l \geq 0}   \frac{\log q}{q^{l+1}}\frac{(2k-1)^{2l+1}}{(2l+1)!} \frac
 {\lambda^{2l+2}_f(q)}{g(q^{2l+1})}\Big)\prod_{i=1}^{R} \Big( \sum_{(n_i, q)=1  } \frac{|\widetilde{\lambda}_f(n_i)\widetilde{\lambda}_f((n_i)_1)|}{\sqrt{n_i (n_{i})_{1}}}
\frac{(2k-1)^{\Omega(n_i)}}{w(n_i) } \widetilde{b}_{i, l}(n_i) \frac {1}{g(n_{i})} \Big),
\end{split}
\end{align}
  where we define $\widetilde{b}_{i, l}(n_i)=b_i(n_iq^{l})$ for the unique index $i$ ($1 \leq i \leq R$) such that $b_i(q) \neq 0$ and
  $\widetilde{b}_{i, l}(n_i)=b_{i}(n_i)$ otherwise. \newline

   We fix a prime $q \in P_{i_0}$ to consider the sum
\begin{align*}
\begin{split}
 S_q:= &  \sum_{l \geq 0}  \Big( \frac{\log q}{q^{l+1}}\frac{(2k-1)^{2l+1}}{(2l+1)!} \frac {\lambda^{2l+2}_f(q)}{g(q^{2l+1})}\Big)\prod_{i=1}^{R} \Big( \sum_{(n_i, q)=1  } \frac{\widetilde{\lambda}_f(n_i)\widetilde{\lambda}_f((n_i)_1)}{\sqrt{n_i (n_{i})_{1}}}
\frac{(2k-1)^{\Omega(n_i)}}{w(n_i) } \widetilde{b}_{i, l}(n_i) \frac {1}{g(n_{i})} \Big) \\
=& \prod_{\substack{i=1 \\ i \neq i_0}}^{R} \Big( \sum_{n_i} \frac{\widetilde{\lambda}_f(n_i)\widetilde{\lambda}_f((n_i)_1)}{\sqrt{n_i (n_{i})_{1}}}
\frac{(2k-1)^{\Omega(n_i)}}{w(n_i) } b_{i}(n_i) \frac {1}{g(n_{i})} \Big) \\
& \hspace*{1in} \times \sum_{l \geq 0}  \Big( \frac{\log q}{q^{l+1}}\frac{(1-2k)^{2l+1}}{(2l+1)!} \frac {\lambda^{2l+2}_f(q)}{g(q^{2l+1})} \sum_{(n_{i_0}, q)=1  } \frac{\widetilde{\lambda}_f(n_{i_0})\widetilde{\lambda}_f((n_{i_0})_1)}{\sqrt{n_{i_0} (n_{i_0})_{1}}}
\frac{(2k-1)^{\Omega(n_{i_0})}}{w(n_{i_0}) } \widetilde{b}_{i_0, l}(n_i) \frac {1}{g(n_{i_0})} \Big).
\end{split}
\end{align*}

  As above, if we remove the restriction of $\widetilde{b}_{i, l}$ on $\Omega(n_i)$, then the sum over $n_i$ becomes
\begin{align*}
\begin{split}
 \prod_{\substack{p \in P_i \\ (p,q)=1}}  \Big( \sum_{m=0}^{\infty} \frac{\lambda^{2m}_f(p)}{p^m} & \frac{(2k-1)^{2m}}{(2m)!g(p^{2m})}  + \sum_{m=0}^{\infty}
\frac{\lambda^{2m+2}_f(p)}{p^{m+1}}
\frac{(2k-1)^{2m+1}}{(2m+1)!}\frac {1}{g(p^{2m+1})}\Big) \\
& \leq \exp \Big ( \Big(\frac {(2k-1)^{2}}{2}+2 k-1 \Big )\sum_{\substack{p\in P_j}}\frac {\lambda^{2}_f(p)}p \Big )E_i,
\end{split}
\end{align*}
 where the last estimation above follows \eqref{1+xexpbound} and
 \[ E_i= \prod_{\substack{p\in P_i }} \Big (1+O\Big( \frac 1{p^2} \Big)\Big )>0. \]

Similar to our discussions above, we see that the error introduced in this process is
\begin{align*}
\begin{split}
\leq \left\{
 \begin{array}
  [c]{ll}
2^{-\ell_i/2}\displaystyle  \exp \Big ( \Big(\frac {(2k-1)^{2}}{2}+2 k-1 \Big )\sum_{\substack{p\in P_i}}\frac {\lambda^{2}_f(p)}p \Big )E_i, \quad i \neq i_0, \\ \\
2^{l-\ell_i/2}\displaystyle \exp \Big ( \Big(\frac {(2k-1)^{2}}{2}+2 k-1 \Big )\sum_{\substack{p\in P_i}}\frac {\lambda^{2}_f(p)}p \Big )E_i, \quad i =i_0.
\end{array}
\right .
\end{split}
\end{align*}

   We deduce from this that
\begin{align*}
\begin{split}
 S_q \leq &  \prod_{i=1}^{R} \exp \Big ( \Big(\frac {(2k-1)^{2}}{2}+2 k-1 \Big )\sum_{\substack{p\in P_i}}\frac {\lambda^{2}_f(p)}p \Big )E_i \prod_{\substack{i=1 \\ i \neq i_0}}^{R} \big (1+2^{-\ell_i/2} \big )\sum_{l \geq 0}  \Big( \frac{\log q}{q^{l+1}}\frac{(2k-1)^{2l+1}}{(2l+1)!} \frac {\lambda^{2l+2}_f(q)}{g(q^{2l+1})}(1+2^{l-\ell_i/2}) \Big ).
\end{split}
\end{align*}

  Applying the bound $|\lambda_f(p)| \leq 2$, Lemma \ref{RS} implies that for some constant $A$ depending on $k$ only,
\begin{align}
\label{Sqbound}
\begin{split}
 S_q \leq & A \Big ( \frac{\log q}{q}+O \Big( \frac{\log q}{q^2} \Big) \Big )  \exp \Big ( \Big(\frac {(2k-1)^{2}}{2}+2 k-1 \Big )\sum_{\substack{p\in P_i}}\frac {\lambda^{2}_f(p)}p \Big ).
\end{split}
\end{align}

  We then conclude from \eqref{6.1}, \eqref{Sqbound} and Lemma \ref{RS} that
\begin{align}
\label{S2bound}
\begin{split}
 S_2 \leq & A X  \exp \Big ( \Big(\frac {(2k-1)^{2}}{2}+2 k-1 \Big )\sum_{\substack{p\in P_i}}\frac {\lambda^{2}_f(p)}p \Big )
  \sum_{q \in \bigcup P_j} \Big (  \frac{\log q}{q}+O \Big( \frac{\log q}{q^2} \Big) \Big ) \\
 \leq & A X \exp \Big ( \Big(\frac {(2k-1)^{2}}{2}+2 k-1 \Big )\sum_{\substack{p\in P_i}}\frac {\lambda^{2}_f(p)}p \Big )  \Big (\frac
 {\log X}{10^{2M}}+O(1) \Big ).
\end{split}
\end{align}

   Combining \eqref{S1bound} and \eqref{S2bound}, we deduce that, upon taking $M$ large enough,
\begin{align*}
\begin{split}
 S_1-S_2 \gg & X \log X \prod_{i=1}^{R}  \exp \Big ( \Big(\frac {(2k-1)^{2}}{2}+2 k-1 \Big )\sum_{\substack{p\in P_i}}\frac {\lambda^{2}_f(p)}p \Big ).
\end{split}
\end{align*}
  The proof of the proposition now follows from this and Lemma \ref{RS}.

\section{Proof of Proposition \ref{Prop5}}

   We apply Lemma \ref{lem1} to get that
\begin{align}
\label{upperboundprodofN}
\begin{split}
 \sumstar_{(d,2)=1}  \mathcal{N}(d, 2k-1)^{2k/(2k-1)} \Phi\Big( \frac dX \Big) \le &
\sumstar_{(d,2)=1}  \Big ( \prod^R_{j=1} \Big ( {\mathcal N}_j(d, 2k) \frac{\left( 1+e^{-\ell_j} \right )^{2k/(2k-1)}}{\left( 1-e^{-\ell_j} \right )^2} +   {\mathcal Q}_j(d)  \Big )\Big
)\Phi \Big( \frac dX \Big) \\
\leq & \prod^R_{j=1} \max \left( \frac{\left( 1+e^{-\ell_j} \right )^{2k/(2k-1)}}{\left( 1-e^{-\ell_j} \right )^2}, 1 \right) \sumstar_{(d,2)=1}  \prod^R_{j=1} \Big ( {\mathcal N}_j(d, 2k) +   {\mathcal Q}_j(d)  \Big )\Phi \Big(\frac dX \Big) \\
\ll & \sumstar_{(d,2)=1}  \prod^R_{j=1} \Big ( {\mathcal N}_j(d, 2k) +   {\mathcal Q}_j(d)  \Big )\Phi\Big(\frac dX \Big),
\end{split}
\end{align}
  where the last estimation above follows by noting that
\begin{align*}
\begin{split}
  \prod^R_{j=1} \max \left( \frac{\left( 1+e^{-\ell_j} \right )^{2k/(2k-1)}}{\left( 1-e^{-\ell_j} \right )^2}, 1 \right) \ll 1.
\end{split}
\end{align*}

  We use the notations in Section \ref{sec 4} and write for $1\le j\le R$,
\begin{align*}
 {\mathcal N}_j(d, 2k-1)=& \sum_{n_j} \frac{\widetilde{\lambda}_f(n_j)}{\sqrt{n_j}} \frac{(2k-1)^{\Omega(n_j)}}{w(n_j)}  b_j(n_j)\chi_{8d}(n_j) \quad \mbox{and} \quad
  {\mathcal P}_j(d)^{2r_k\ell_j} =&  \sum_{ \substack{ \Omega(n_j) = 2r_k\ell_j \\ p|n_j \implies  p\in P_j}}
  \frac{\widetilde{\lambda}_f(n_j)}{\sqrt{n_j}}\frac{(2r_k\ell_j)! }{w(n_j)}\chi_{8d}(n_j),
\end{align*}
  where $r_k=1+\lceil \frac {k}{2k-1} \rceil$. \newline

  Applying the bound
\begin{align}
\label{Stirling}
  \Big( \frac ne \Big)^n \leq n! \leq n \Big( \frac ne \Big)^n,
\end{align}
 we deduce that
\begin{align*}
 & \Big( \frac{124 k^2}{\ell_j} \Big)^{2r_k\ell_j} \leq (2r_k \ell_j)!
\leq  2r_k\ell_j\Big( \frac{248 k^2 r_k }{e} \Big)^{2r_k\ell_j}.
\end{align*}

  Note further we have $\widetilde{\lambda}_f(n_j) \leq 2^{\Omega(n_j)}$ as $\lambda_f(p) \leq 2$ and $b_j(n_j)$ restricts $n_j$ to satisfy $\Omega(n_j) \leq \ell_j$. It follows from the above discussions that we can write ${\mathcal N}_j(d, 2k) + {\mathcal Q}_j(d)$
as a Dirichlet polynomial of the form
\begin{align*}
  D_j(d)=\sum_{n_j \leq X^{2r_k/\ell_j}}\frac{a_{n_j}b_j(n_j)}{\sqrt{n_j}}\chi_{8d}(n_j)
\end{align*}
   where for some constant $B(k)$ whose value depends on $k$ only,
\begin{align*}
  |a_{n_j}| \leq B(k)^{\ell_j}.
\end{align*}
    We then apply Lemma \ref{PropDirpoly} to evaluate the last expression in \eqref{upperboundprodofN} and deduce from \eqref{sumoverell} that, for large $M$ and $N$, the contribution arising from the error term in \eqref{meancharsum} is
\begin{align*}
 \ll B(k)^{\sum^R_{j=1}\ell_j}X^{1/2+\varepsilon}X^{2r_k\sum^R_{j=1} 1/\ell_j} \ll B(k)^{R\ell_1}X^{1/2+\varepsilon}X^{4r_k/\ell_R} \ll X^{1-\varepsilon}.
\end{align*}

   We may thus focus on the contributions arising from the main term in \eqref{meancharsum}.  Thus the last expression in \eqref{upperboundprodofN} is
\begin{align}
\label{maintermbound}
\begin{split}
 \ll & X \times \sum_{\substack{\prod^R_{j=1}n_j=\square}} \Big ( \prod^R_{j=1}\frac{a_{n_j}b_j(n_j)}{\sqrt{n_j}} \times \prod_{p | \prod^R_{j=1}n_j}\Big ( \frac p{p+1} \Big ) \Big ) = X \times \prod^R_{j=1} \sum_{n_j=\square} \Big ( \frac{a_{n_j}b_j(n_j)}{\sqrt{n_j}}  \times \prod_{p | n_j}\Big ( \frac p{p+1} \Big ) \Big ),
\end{split}
\end{align}
where the summation condition $m=\square$ restricts the summand $m$ to squares.  We note that
\begin{align}
\label{sqinN}
\begin{split}
   \sum_{n_j=\square} \frac{a_{n_j}b_j(n_j)}{\sqrt{n_j}}  \prod_{p | n_j}\Big ( \frac p{p+1} \Big )
= \sum_{n_j=\square} \frac{\widetilde{\lambda}_f(n_j)}{\sqrt{n_j}} & \frac{(2k)^{\Omega(n_j)}}{w(n_j)}  b_j(n_j) \prod_{p | n_j}\Big ( \frac p{p+1} \Big ) \\
& +\Big( \frac{124k^2 }{\ell_j} \Big)^{2r_k\ell_j}(2r_k\ell_j)!\sum_{ \substack{ n_j=\square \\ \Omega(n_j) = 2r_k\ell_j \\ p|n_j \implies
p\in P_j}} \frac{\widetilde{\lambda}_f(n_j)}{\sqrt{n_j}}\frac{1 }{w(n_j)}\prod_{p | n_j}\Big ( \frac p{p+1} \Big ).
\end{split}
\end{align}

  Now, using $|\lambda_f(p)| \leq 2$, we obtain
\begin{align}
\label{sqinN'}
\begin{split}
  \sum_{n_j=\square}  \frac{\widetilde{\lambda}_f(n_j)}{\sqrt{n_j}} & \frac{(2k)^{\Omega(n_j)}}{w(n_j)}  b_j(n_j) \prod_{p | n_j}\Big ( \frac p{p+1} \Big )
\leq  \prod_{p \in P_j}\Big ( 1 + \frac {\lambda^2_f(p)\big(2k \big )^2}{2p}\Big ( \frac p{p+1} \Big )+\sum_{i \geq 2}\frac {\lambda_f(p)^{2i}(2k)^{2i}}{p^{i}(2i)!}\Big (
\frac p{p+1} \Big )\Big ) \\
& \leq \prod_{p \in P_j}\Big ( 1 + \frac {\lambda^{2}_f(p)\big(2k \big )^2}{2p}\Big ( \frac p{p+1} \Big )+\frac {e^{4k}}{p^{2}}\Big )
\leq \exp \Big (\frac {\big(2k \big )^2}{2}\sum_{p \in P_j}\frac {\lambda^{2}_f(p)}{p}+\sum_{p \in P_j}\frac {e^{4k}}{p^{2}}\Big ),
\end{split}
\end{align}
 where \eqref{1+xexpbound} is again utilized to obtain the last bound above. \newline

  Upon replacing $n$ by $n^2$ and noting that $\widetilde{\lambda}_f(n^2_j)=\widetilde{\lambda}^2_f(n_j), 1/w(n^2) \leq 1/w(n)$, we deduce that
\begin{align*}
 &  \Big( \frac{124k^2 }{\ell_j} \Big)^{2r_k\ell_j}(2r_k\ell_j)!\sum_{ \substack{ n_j=\square \\ \Omega(n_j) = 2r_k\ell_j \\ p|n_j \implies
p\in P_j}} \frac{\widetilde{\lambda}_f(n_j)}{\sqrt{n_j}}\frac{1 }{w(n_j)}\prod_{p | n_j}\Big ( \frac p{p+1} \Big )
\leq  \Big( \frac{124k^2 }{\ell_j} \Big)^{2r_k\ell_j}
\frac {(2r_k\ell_j)!}{ (r_k\ell_j)!}\Big (\sum_{p \in P_j}\frac{\lambda^2_f(p)}{p} \Big )^{r_k\ell_j}.
\end{align*}
Applying \eqref{boundsforsumoverp}  and \eqref{Stirling} to estimate the right side expression above to get that for some constant $B_1(k)$
   depending on $k$ only,
 \begin{align*}
 \Big( \frac{124k^2 }{\ell_j} \Big)^{2r_k\ell_j} & (2r_k\ell_j)!\sum_{ \substack{ n_j=\square \\ \Omega(n_j) = 2r_k\ell_j \\ p|n_j \implies
p\in P_j}} \frac{\widetilde{\lambda}_f(n_j)}{\sqrt{n_j}}\frac{1 }{w(n_j)}\prod_{p | n_j}\Big ( \frac p{p+1} \Big ) \ll B_1(k)^{\ell_j}e^{-r_k\ell_j\log (r_k\ell_j)}\Big (\sum_{p \in P_j}\frac {\lambda^{2}_f(p)}{p} \Big )^{r_k\ell_j}
\\
& \ll  B_1(k)^{\ell_j}e^{-r_k\ell_j\log (r_k\ell_j)}e^{r_k\ell_j \log (2\ell_j/N)} \ll e^{-\ell_j}\exp \Big (\frac {(2k)^2}{2}\sum_{p \in P_j}\frac {\lambda^{2}_f(p)}{p}  \Big ).
\end{align*}

  Combining the above with \eqref{maintermbound}--\eqref{sqinN'}, we get that the expression in \eqref{upperboundprodofN} is
\begin{align*}
\begin{split}
\ll & X \prod^R_{j=1} \Big (1+e^{-\ell_j} \Big ) \exp \Big (\frac {\big(2k \big )^2}{2}\sum_{p \in P_j}\frac {\lambda^{2}_f(p)}{p} +\sum_{p \in P_j}\frac
{e^{4k}}{p^{2}}\Big ) \ll X ( \log X  )^{\frac {(2k)^2}{2}},
\end{split}
\end{align*}
where the last bound emerges by using Lemma \ref{RS}. The assertion of the proposition now follows from this.

\vspace*{.5cm}

\noindent{\bf Acknowledgments.}  P. G. is supported in part by NSFC grant 11871082 and L. Z. by the FRG grant PS43707 at the University of New South Wales.  The authors would like to thank the anonymous referee for his/her meticulous inspection of the paper and many helpful comments.

\bibliography{biblio}
\bibliographystyle{amsxport}

\vspace*{.5cm}

\end{document}